\documentclass[12pt,a4paper]{amsart}

\usepackage[utf8]{inputenc}

\usepackage{mathrsfs}
\usepackage{todonotes}

\usepackage{amssymb, tikz}

\usepackage{hyperref}

\usepackage{amsmath}

\newtheorem{theorem}{Theorem}[section]

\newtheorem{lemma}[theorem]{Lemma}

\theoremstyle{definition}

\theoremstyle{remark}
\newtheorem{remark}[theorem]{Remark}
\newtheorem{question}[theorem]{Question}

\newcommand{\Ord}{{\rm Ord}}

\newcommand{\pcf}{{\rm pcf}}
\newcommand{\Reg}{{\rm Reg}}
\newcommand{\Add}{{\rm Add}}

\newcommand{\bbR}{{\mathbb R}}

\newcommand{\bbP}{{\mathbb P}}

\newcommand{\bbQ}{{\mathbb Q}}

\newcommand{\cf}{{\rm cf}}

\newcount\skewfactor
\def\mathunderaccent#1#2 {\let\theaccent#1\skewfactor#2
\mathpalette\putaccentunder}
\def\putaccentunder#1#2{\oalign{$#1#2$\crcr\hidewidth
\vbox to.2ex{\hbox{$#1\skew\skewfactor\theaccent{}$}\vss}\hidewidth}}
\def\name{\mathunderaccent\tilde-3 }

\begin{document}

\title[Unlimited
accumulation by the Shelah's PCF operator]{Unlimited
accumulation by the Shelah's PCF operator}

\author {Mohammad Golshani}
\address{School of Mathematics\\
 Institute for Research in Fundamental Sciences (IPM)\\
  P.O. Box:
19395-5746\\
 Tehran-Iran.}
\email{golshani.m@gmail.com}
\urladdr{http://math.ipm.ac.ir/~golshani/}

\thanks{The  author's research has been supported by a grant from IPM (No. 1401030417). He thanks the referee of the paper
for his many useful remarks and suggestions that improved the presentation of the paper.}





\begin{abstract}
 Modulo the existence of large cardinals, there is a model of set theory in which for some set $B$ of regular cardinals,
 the sequence  $\langle  \pcf^\alpha(B): \alpha \in \Ord     \rangle$ is strictly increasing. The result answers a question from \cite{kenta}.
\end{abstract}

\maketitle
\setcounter{section}{-1}


\section{Introduction}
Pcf theory was developed by Shelah, see \cite{shelah},  to answer some cardinal arithmetic questions, though it later found wider applications in
set theory and other parts of mathematics. A key concept in pcf theory is that of $\pcf(A)$, for a set $A$ of regular cardinals, which is defined as
\[
\pcf(A)=\{\cf(\prod A, <_D): D \text{~an ultrafilter on~} A   \}.
\]
If the set $A$ is an interval of regular cardinals and is progressive, meaning that $\min(A) > |A|$, then the set $\pcf(A)$ behaves nicely,
 in particular, it satisfies the following closure property:
\[
\pcf(\pcf(A))=\pcf(A).
 \]
 In the absence of such assumptions on the set $A$, the set
$\pcf(A)$ lacks some of its nice properties, see for example \cite{jech} and \cite{kenta}.

Let $\Reg$ denote the class of regular cardinals. For a set $A \subseteq \Reg$ let us define the increasing sequence
 $\langle  \pcf^\alpha(A): \alpha \in \Ord     \rangle$ by induction on $\alpha$ as follows:
\begin{center}
 $\pcf^{\alpha}(A)=$ $\left\{
\begin{array}{l}
         A \hspace{3.5cm} \text{ if } \alpha=0,\\
         \pcf(\pcf^\beta(A))  \hspace{1.5cm} \text{ if } \alpha=\beta+1,\\
         \bigcup_{\beta < \alpha} \pcf^\beta(A) \hspace{1.4cm} \text{ if } \lim(\alpha).
     \end{array} \right.$
\end{center}
In \cite{kenta}, Tsukuura showed that if there are infinitely many supercompact cardinals, then in a forcing extension, for a set $A$ of regular cardinals, the sequence $\langle  \pcf^n(A): n<\omega     \rangle$
can be strictly increasing and he raised the following question.

\begin{question}
\label{a1} Is it a theorem of ZFC that for every $A \subseteq \Reg$ there is an $\alpha$ such that
$\pcf^{\alpha+1}(A)=\pcf^\alpha(A)$.
\end{question}
In this paper we give a consistent negative answer to this question, by proving the following theorem.
\begin{theorem}
\label{a2} Assume $\langle \kappa_n: n<\omega \rangle $ is an increasing sequence of  supercompact  cardinals  and  the set $A=\{\alpha < \kappa_0: \alpha$ is $< \kappa_0$-supercompact    $ \}$ belongs to any normal measure on $\kappa_0$. Then there exist  a model $\bold M$ of set theory and an inaccessible cardinal $\theta$, such that in $\bold M,$ the sequence
$\langle  \pcf^\alpha(\theta \cap \Reg): \alpha \in \Ord     \rangle$
is strictly increasing.
\end{theorem}
\begin{remark}
Note that the set $\theta \cap \Reg$ is an interval of regular cardinals, however $|\theta \cap \Reg|=\theta > \min(\theta \cap \Reg)$, so $\theta \cap \Reg$ is not progressive.
\end{remark}
The basic idea of the proof is  to find a forcing extension of  the universe in which $\kappa_0$ remains inaccessible, and for all large enough cardinals
$\lambda < \kappa,$  $2^\lambda$ was measurable in the ground model, and its measurability is further preserved by some suitably forcing notions.
The main technical idea of the proof is the use of Radin forcing as developed by Foreman and Woodin \cite{foreman}. This makes the construction works by allowing
the process to proceed at limit points.

Our notation is standard. Given a forcing notion $\bbP$, let $\text{RO}(\bbP)$ denote its Boolean completion.  We also use the symbol $\simeq$ for the equivalence of forcing notions, thus given two forcing notions
$\bbP$ and $\bbQ$,
$\bbP \simeq \bbQ$ means that $\text{RO}(\bbP)$ is isomorphic to  $\text{RO}(\bbQ)$.

\section{Some preliminaries}
\label{s1}
In this section we present a few lemmas that we will need for the proof of
Theorem \ref{a2}. The first result is Laver's indestructibility result for supercompact cardinals.

\begin{lemma} (Laver \cite{laver})
\label{a3} Suppose $\kappa$ is a supercompact cardinal and $\lambda < \kappa$.
Then there exists a $\lambda^+$-directed closed $\kappa$-c.c. forcing notion
$\mathbb{L}(\lambda, \kappa)$ of size $\kappa$
such that in $V^{\mathbb{L}(\lambda, \kappa)}$, the supercompactness of $\kappa$
remains indestructible under $\kappa$-directed closed forcing notions.
\end{lemma}
The next lemma shows that for a measurable cardinal $\kappa$, the set $\pcf(\kappa \cap \Reg)$ is already large and
also discusses the structure of $\pcf(A)$ in nice forcing extensions.
\begin{lemma} (see \cite{kenta})
\label{a4}
\begin{enumerate}
\item  If $\kappa$ is a measurable cardinal, then $\pcf(\kappa \cap \Reg)= (2^\kappa)^+ \cap \Reg$,

\item Suppose $A \subseteq \Reg$, $\lambda=\min(A)$ and $\bbQ$ is a $\lambda$-c.c.c forcing notion. Then
$\pcf^V(A) \subseteq \pcf^{V[\bold G_{\bbQ}]}(A).$
\end{enumerate}
\end{lemma}
The next lemma is due to Shelah.
\begin{lemma} (see \cite{shelah})
\label{a5} Suppose $\lambda$ is a singular cardinal. Then $\lambda^+ \in \pcf(\lambda \cap \Reg)$.
\end{lemma}

\section{Proof of Theorem \ref{a2}}
In this section we prove our main result by building a model of ZFC in which for some set $B$ of regular cardinals, the sequence
$\langle  \pcf^\alpha(B): \alpha \in \Ord     \rangle$ is strictly increasing.

Thus suppose that $\text{GCH}$ holds and  $\langle \kappa_n: n<\omega   \rangle $
is an increasing sequence of supercompact cardinals. For notational simplicity set
 $\kappa=\kappa_0$. Let also $A=\{\alpha < \kappa: \alpha$ is $< \kappa$-supercompact    $ \}$, and
suppose that $A$ belongs to any normal measure on $\kappa$.
 \\
 {\bf Step 1.} At the first step, we build a forcing extension in which each $\kappa_n$ becomes indestructibly supercompact,
 and the elements of the set $A$ remain partially supercompact and indestructible. To make the cardinals $\kappa_n$ indestructibly
 supercompact, we use the Laver forcing \ref{a3}, and in order to obtain indestructibility for partially supercompact cardinals which are in
 $A$, we use \cite{apter}.
Thus by a suitable preparation, there exists a full support iteration of the form
 \[
{\bbP}^1= \langle \langle  \bbP^1_n: n \leq \omega       \rangle, \langle  \name{\bbQ}^1_n: n<\omega     \rangle\rangle,
\]
where $\bbP^1_0$ is the trivial forcing and:
\begin{enumerate}
\item[(1-1)] $\bbP^1_1 \simeq \bbQ^1_0 \subseteq V_\kappa$ is an iteration  such that in the generic extension by it,
the set
$A'=\{\alpha < \kappa: \alpha$ is $< \kappa$-supercompact and indestructible under $\alpha$-directed closed forcing notions from
$V_\kappa \}$
belongs to any normal measure on $\kappa$,

\item[(1-2)] after forcing with  $\bbP^1_1,$ $\kappa$ becomes supercompact and indestructible and $2^\kappa=\kappa^+$,

\item[(1-3)] for $n<\omega, $  $\Vdash_{\bbP^1_n}$``$\name{\bbQ}^1_{n+1}=\mathbb{L}(\kappa_{n}, \kappa_{n+1})$'', where $\mathbb{L}(\kappa_{n}, \kappa_{n+1})$ is as in Lemma \ref{a3}.
\end{enumerate}

Let $\bold V^1=\bold V[\bold G_{\bbP^1}]=\bold V[\bold G_{\bbP^1_\omega}]$. Thus in $\bold V^1$,
each $\kappa_n, n<\omega$, is supercompact and Laver indestructible and  $2^{\kappa_n}=\kappa_n^+$. Furthermore, the set
$A'$ as defined in (1-1), belongs to any normal measure on $\kappa$.
\\
{\bf Step 2.} Work in $\bold V^1$ and force with the full support product
$\bbP^2=\prod_{n<\omega} \Add(\kappa_n, \kappa_{n+1})$. In the generic extension $\bold V^2=\bold V^1[\bold G_{\bbP^2}]$, the cardinal $\kappa$
still remains supercompact and  $2^{\kappa_n}=\kappa_{n+1}$ for each  $n<\omega$.
Note that
\[
\bbP^1 \ast \name{\bbP}^2 = \bbP^1_1 \ast (\name{\bbP}^1_\omega / \dot{\bold{G}}_{\bbP^1_1}) \ast (\prod_{n<\omega} \name{\Add}(\kappa_n, \kappa_{n+1})),
\]
where $\bbP^1_1 \subseteq V_\kappa$ and the rest of the forcing is $\kappa$-directed closed.
Thus for a measure one set $A''$
 of $\beta < \kappa$ we can assume that:
\begin{enumerate}
\item[(2-1)] $\beta$ is $<\kappa$-supercompact in $\bold V,$

\item[(2-2)]  $\bbP^1_1$ can be written  as
\[
\bbP^1_1 \simeq \bbP^{1}_1(< \beta) \ast \name{\bbP}^1_1(\beta) \ast \name{\bbP}^1_{1}(> \beta),
\]
where
\begin{enumerate}
\item $\bbP^1_1(<\beta)$ is $\beta$-c.c. of size $\beta$,

\item $\bbP^1_1(< \beta)$ forces $\beta$ is $< \kappa$-supercompact and indestructible under $\beta$-directed closed forcing notions from
$V_\kappa$,
\item $\Vdash_{\bbP^1_1(<\beta)}$``$\name{\bbP}^1_1(\beta)=(\name{\bbP}^1_{1,\omega}(\beta) / \dot{\bold{G}}_{\name{\bbP}^1_1(<\beta)}) \ast \prod_{n<\omega}\name{\Add}(\zeta_{\beta, n}, \zeta_{\beta, n+1})$'',
where $\langle \zeta_{\beta, n}: n<\omega     \rangle$ is an increasing sequence of elements of $A$ with $\zeta_{\beta, 0}=\beta$, and such that working in the generic extension by $\bbP^1_1(<\beta)$, the forcing notion
$$\bbP^1_{1,\omega}(\beta)/\bold{G}_{{\bbP}^1_1(<\beta)}= \langle \langle  {\bbP}^1_{1,n}(\beta): 0< n \leq \omega       \rangle, \langle  \name{\bbQ}^1_{1,n}(\beta): 0< n<\omega     \rangle\rangle$$
 is the full support iteration
of the forcing notions $\name{\bbQ}^1_{1, n+1}(\beta)=\name{\mathbb{L}}(\zeta_{\beta, n}, \zeta_{\beta, n+1})$''.

\item $\Vdash_{\bbP^1_1(<\beta) \ast \name{\bbP}^1_{1,n}(\beta) / \dot{\bold{G}}_{\name{\bbP}^1_1(\beta)}}$``the measurability of $\zeta_{\beta, n}$ is indestructible under $\zeta_{\beta, n}$-directed closed forcing notions from $V_\kappa$'',


\item $\Vdash_{\bbP^{1}_1(< \beta) \ast \name{\bbP}^1_1(\beta)}$``$\name{\bbP}^1_{1}(> \beta)$ is
$\mu$-directed closed, where $\mu \in A$ is above $\sup_{n<\omega}\zeta_{\beta, n}$.
\end{enumerate}
\end{enumerate}
Note that this is true for $\kappa$ and the sequence $\langle  \kappa_n: n<\omega   \rangle$,
and so by supercompactness of $\kappa$, those properties reflect down to a large set.
\\
{\bf Step 3.} Working in the model $\bold V^2$, by \cite{foreman},
there exists a forcing notion $\bbR$ such that the following properties are hold:
\begin{enumerate}
\item[(3-1)] $(\bbR, \leq, \leq^*)$ is a Prikry type forcing notion,

\item[(3-2)] $\bbR$ is $\kappa^+$-c.c.,

\item[(3-3)] $\kappa$ remains inaccessible in $\bold V^2[\bold G_{\bbR}],$

\item[(3-4)] forcing with $\bbR$ adds a club $C$ of $\kappa$ consisting of elements of $A''$,\footnote{This can always be done, by forcing below a condition,  which forces each element of the Radin club to be in $A''$.}

\item[(3-5)] suppose $\alpha < \beta$ are two successive points in $C$. Then $\beth_5(\alpha)=2^{\beth_4(\alpha)}=\beta,$

\item[(3-6)] (factorization property)\footnote{Indeed one can say more. Given a condition $p \in \bbR$ such that
$\alpha$ and $\beta$ appear as successive points in $p$, one can factor $\bbR/p$ as
$\big(\bbR^{< \alpha}/ p(<\alpha)\big) \times \big(\prod_{i<4} \Add(\beth_i(\alpha), \beth_{i+1}(\alpha)) \times \Add(\beth_4(\alpha), \beta) \big) \times \big(\bbR^{>\beta} / p(>\beta) \big)$. Note that this representation does not depend on the club $C$, but once we have $C$,
given $\alpha \in C$, $\beta =\min(C \setminus (\alpha+1))$ and $p \in \bold G$, we can always find an extension $q$ of $p$ such that $q \in \bold G_{\bbR}$ and $\alpha$
and $\beta$ appear in $q$.} suppose $\alpha \in C$ and $\beta =\min(C \setminus (\alpha+1))$. Then
\[
\bbR \simeq \bbR^{< \alpha} \times \prod_{i<4} \Add(\beth_i(\alpha), \beth_{i+1}(\alpha)) \times \Add(\beth_4(\alpha), \beta) \times \bbR^{>\beta},
\]
where
\begin{enumerate}
\item $\bbR^{< \alpha}$ is $\alpha^+$-c.c.,

\item $(\bbR^{> \beta}, \leq^*)$ is $\beta^+$-closed.
\end{enumerate}

\end{enumerate}
Let $\bold V^3=\bold V^2[\bold G_{\bbR}]$.
\\
{\bf Step 4.}
We will need the following  lemma.
\begin{lemma} (in $\bold V^3$)
\label{a7} Suppose $\beta$ is a  limit point of $C$. Then
\[
\pcf( \beta \cap \Reg) = \beth_1(\beta)^+ \cap \Reg.
\]
\end{lemma}
The proof is essentially the same as in \cite{jech} ( see also \cite{kenta}).
\\
{\bf Step 5.} In this step we complete the proof of Theorem \ref{a2}.
Let $\langle  \theta_\xi: \xi < \kappa    \rangle$
be an increasing enumeration of $C$, where $C$ is the club added by Radin forcing in step 3, and set $\theta_\kappa=\kappa.$
 We show that the model
 $\bold M= \bold V^3_{\kappa}$ (=the model $\bold V_\kappa$ as computed in ${\bold V^3}$)
is as required.
To see this, let  $B=\theta_0 \cap \Reg.$ It suffices to show that
\[
\langle  \pcf^\alpha(B): \alpha < \kappa   \rangle
\]
is strictly increasing.

Suppose $\alpha < \kappa$  and write it as $\alpha=\delta+n$, where $\delta < \kappa$ is a limit ordinal and $n<\omega$.
We show by induction that:
\begin{itemize}
\item if $n >0$, then
\[
\pcf^\alpha(B)= \beth_{n}(\theta_{\delta})^+ \cap \Reg.
\]
\item
if $\alpha$ is a limit ordinal (including $\alpha=0$),
then
\[
\pcf^\alpha(B)= \theta_\alpha \cap \Reg.
\]
\end{itemize}
The claim is true for $\alpha=0$,  and if $\alpha$ is a limit ordinal, then by the induction hypothesis,
\[
\pcf^\alpha(B)=  \bigcup_{\xi < \alpha, \lim(\xi)}\bigcup_{n<\omega} \pcf^{\xi+n}(B) = \bigcup_{\xi < \alpha, \lim(\xi)}\bigcup_{n<\omega}
\beth_{n}(\theta_{\xi})^+ \cap \Reg=\theta_\alpha \cap \Reg.
\]
Now suppose the claim for $\alpha$. If $\alpha$ is a limit ordinal, then  by Lemma \ref{a7},
\[
\pcf^{\alpha+1}(B) =\pcf(\theta_\alpha \cap \Reg) = \beth_{1}(\theta_{\alpha})^+ \cap \Reg.
\]
Otherwise $\alpha=\delta+n$ for some $n>0$
and using Lemma \ref{a4}, and the factorization properties (2-2)  and (3-6) we can easily conclude that
\[
\pcf^{\alpha+1}(B) =\pcf(\beth_{n}(\theta_{\delta})^+ \cap \Reg) = \beth_{n+1}(\theta_{\delta})^+ \cap \Reg.
\]
Indeed, let $n=5m+k$, where $k<5$.  Then $\beth_n(\theta_\delta)=\beth_k(\theta_{\delta+m})$. Let $\beta=\theta_{\delta+m}$, so
\begin{itemize}
\item $\beta \in C,$
 \item for $i \leq 4$, $\beth_{i}(\theta_{\delta+m})=\zeta_{\beta, i}$, \footnote{See (2-2)(c).}
 \item $\beth_{5}(\theta_{\delta+m})=\theta_{\delta+m+1}$,
\end{itemize}
 To continue the proof, we consider two cases depending on whether $k<4$ or $k=4$.
 Let us  suppose that $k<4,$ the case $k=4$ can be treated in a similar way.
\begin{lemma}
\label{lem1} The forcing notion $\bbP^1 \ast \name{\bbP}^2 \ast \name{\bbR}$ can be factored as
 \[
 \bbP^1 \ast \name{\bbP}^2 \ast \name{\bbR} \simeq \mathbb{L} \ast  \name{\mathbb{K}} \ast  \name{\mathbb{J}},
 \]
 where
 \begin{enumerate}
 \item in the extension by $\mathbb{L}$, the cardinal $\zeta_{\beta, k}$ is measurable and its measurability is indestructible under
 $\zeta_{\beta, k}$-closed forcing notions,

\item 
  ${\mathbb{K}} \simeq  \big(\prod_{j \leq k}{\Add}(\zeta_{\beta, j}, \zeta_{\beta, j+1}) \big) \ast \bigg(\name{\bbR}^{< \theta_{\delta+m}} \times \prod_{i \leq k} \Add(\zeta_{\beta, i}, \zeta_{\beta, i+1})\bigg),$

\item  $\name{\mathbb{J}}$ is forced to add no new subsets to $\zeta_{\beta, k}^+.$
  \end{enumerate}
 \end{lemma}
 \begin{proof}
 Recall from  (2-2) that $\bbP^1_1 \simeq \bbP^{1}_1(< \beta) \ast \name{\bbP}^1_1(\beta) \ast \name{\bbP}^1_{1}(> \beta)$,
 where
 \begin{center}
 $\Vdash_{\bbP^{1}_1(< \beta) \ast \name{\bbP}^1_1(\beta)}$``$\name{\bbP}^1_{1}(> \beta)$ is $\zeta_{\beta, k}^+$-directed closed'',\footnote{Indeed, it is forced to have more closure properties, but  $\zeta_{\beta, k}^+$-directed closedness is sufficient for us.}
 \end{center}
 and the forcing notion $\bbP^1_1(\beta)$ is forced to be of the form
$$
 \bbP^1_1(\beta) = (\name{\bbP}^1_{1,\omega}(\beta) / \dot{\bold{G}}_{\name{\bbP}^1_1(<\beta)}) \ast \prod_{n<\omega}\name{\Add}(\zeta_{\beta, n}, \zeta_{\beta, n+1})$$
 where
$\bbP^1_{1,\omega}(\beta)/\bold{G}_{{\bbP}^1_1(<\beta)}= \langle \langle  {\bbP}^1_{1,n}(\beta): 0< n \leq \omega       \rangle, \langle  \name{\bbQ}^1_{1,n}(\beta): 0< n<\omega     \rangle\rangle$
 is the full support iteration
of the forcing notions $\name{\bbQ}^1_{1, n+1}(\beta)=\name{\mathbb{L}}(\zeta_{\beta, n}, \zeta_{\beta, n+1})$''.
In particular, we can write $\bbP^1_1(\beta)$
as
 $$\bbP^1_1(\beta)=\bbP^1_{1, k}(\beta) \ast \big(\name{\bbP}^1_{1, \omega} / \dot{\bold{G}}_{\bbP^1_{1, k}(\beta)}\big)  \ast \prod_{n<\omega}\name{\Add}(\zeta_{\beta, n}, \zeta_{\beta, n+1}),$$
 where
\begin{enumerate}
\item[(4-1)] $\bbP^1_{1, k}(\beta)$ makes the measurability of $\zeta_{\beta, k}$ indestructible under $\zeta_{\beta, k}$-forcing notions from
$\bold V_\kappa$,

\item[(4-2)] $\Vdash_{\bbP^1_{1, k}(\beta)}$``$\name{\bbP}^1_{1, \omega} / \dot{\bold{G}}_{\bbP^1_{1, k}(\beta)}$ is  $\zeta_{\beta, k}^+$-directed closed''.
\end{enumerate}
Thus,
\[
\bbP^1_1 \simeq \bbP^1_1(<\beta) \ast \bbP^1_{1, k}(\beta) \ast  \bigg(\prod_{j \leq k}\name{\Add}(\zeta_{\beta, j}, \zeta_{\beta, j+1})
\times \name{\bbP}^1_{1, \omega} / \dot{\bold{G}}_{\bbP^1_{1, k}(\beta)} \bigg) \ast \bigg(\prod_{j > k}\name{\Add}(\zeta_{\beta, j}, \zeta_{\beta, j+1})  \bigg) \ast \name{\bbP}^1_1(> \beta),
\]
This is true as $\prod_{j \leq k}\name{\Add}(\zeta_{\beta, j}, \zeta_{\beta, j+1})  $ is computed the same, in the extension by the forcing notions $\bbP^1_{1, k}(\beta)$ and $\bbP^1_{1, k}(\beta) \ast \big(\name{\bbP}^1_{1, \omega} / \dot{\bold{G}}_{\bbP^1_{1, k}(\beta)}\big)$, and
\[
\bbP^1_{1, \omega} / \dot{\bold{G}}_{\bbP^1_{1, k}(\beta)} \ast \prod_{j \leq k}\name{\Add}(\zeta_{\beta, j}, \zeta_{\beta, j+1}) \simeq \bbP^1_{1, \omega} / \dot{\bold{G}}_{\bbP^1_{1, k}(\beta)} \times \prod_{j \leq k}\name{\Add}(\zeta_{\beta, j}, \zeta_{\beta, j+1}).
\]

On the other hand, the forcing notion $\bbP^2$ is $\kappa$-directed closed, and by (3-6),  the forcing
 notion $\bbR$ can be written as
 $$\bbR=\bbR^{< \theta_{\delta+m}} \times \big(\prod_{i<4} \Add(\zeta_{\beta, i}, \zeta_{\beta, i+1}) \times
\Add(\zeta_{\beta, 4}, \theta_{\delta+m+1}) \big)  \times  \bbR^{> \theta_{\delta+m+1}},$$
 where
 \begin{enumerate}
\item[(5-1)] $\bbR^{< \theta_{\delta+m}}$ is $\beta^+$-c.c.,

 \item[(5-2)] forcing with $\bbR^{> \theta_{\delta+m+1}}$ adds no new  subsets of $ \theta_{\delta+m+1}$.
 \end{enumerate}
We have
  $$\bbR=\bbR^{< \theta_{\delta+m}} \times \prod_{i\leq k} \Add(\zeta_{\beta, i}, \zeta_{\beta, i+1}) \times
\prod_{k<i<4} \Add(\zeta_{\beta, i}, \zeta_{\beta, i+1}) \times \Add(\zeta_{\beta, 4}, \theta_{\delta+m+1})   \times  \bbR^{> \theta_{\delta+m+1}}.$$

It follows that $\bbP^1 \ast \name{\bbP}^2 \ast \name{\bbR}$ can be written as
$$\begin{array}{ll}
\bbP^1 \ast \name{\bbP}^2 \ast \name{\bbR}&\simeq \bbP^1_1(<\beta) \ast \bbP^1_{1, k}(\beta) \ast  \bigg(\prod_{j \leq k}\name{\Add}(\zeta_{\beta, j}, \zeta_{\beta, j+1})
\times \name{\bbP}^1_{1, \omega} / \dot{\bold{G}}_{\bbP^1_{1, k}(\beta)} \bigg)\\
&\quad \ast \big(\prod_{j > k}\name{\Add}(\zeta_{\beta, j}, \zeta_{\beta, j+1})  \big) \ast \name{\bbP}^1_1(> \beta) \ast {\bbP}^1_{\omega} / {\dot G}_{\bbP^{1}_{1}}\\
&\quad\ast \prod_{n<\omega} \name{\Add}(\kappa_n, \kappa_{n+1})\\
&\quad\ast \bigg(\bbR^{< \theta_{\delta+m}} \times \big(\prod_{i<4} \Add(\zeta_{\beta, i}, \zeta_{\beta, i+1}) \times
\Add(\zeta_{\beta, 4}, \theta_{\delta+m+1}) \big)  \times  \bbR^{> \theta_{\delta+m+1}}\bigg)\\
&\simeq \bbP^1_1(<\beta) \ast \bbP^1_{1, k}(\beta) \ast  \big(\prod_{j \leq k}\name{\Add}(\zeta_{\beta, j}, \zeta_{\beta, j+1}) \big)\\
&\quad \ast\name{\bbP}^1_{1, \omega} / \dot{\bold{G}}_{\bbP^1_{1, k}(\beta)}  \ast \big(\prod_{j > k}\name{\Add}(\zeta_{\beta, j}, \zeta_{\beta, j+1})  \big) \ast \name{\bbP}^1_1(> \beta) \ast {\bbP}^1_{\omega} / {\dot G}_{\bbP^{1}_{1}}\\
&\quad\ast \prod_{n<\omega} \name{\Add}(\kappa_n, \kappa_{n+1})\\
&\quad\ast \bigg( \bbR^{< \theta_{\delta+m}} \times \prod_{i \leq k} \Add(\zeta_{\beta, i}, \zeta_{\beta, i+1})\\
&\quad \times
\prod_{k<i<4} \Add(\zeta_{\beta, i}, \zeta_{\beta, i+1}) \times \Add(\zeta_{\beta, 4}, \theta_{\delta+m+1})   \times  \bbR^{> \theta_{\delta+m+1}}\bigg)\\
\end{array}$$
Now note that the forcing notion $\bbR^{< \theta_{\delta+m}} \times \prod_{i \leq k} \Add(\zeta_{\beta, i}, \zeta_{\beta, i+1})$
is defined the same in the forcing extensions by $\bbP^1 \ast \bbP^2$
and $\bbP^1_1(<\beta) \ast \bbP^1_{1, k}(\beta) \ast  \big(\prod_{j \leq k}\name{\Add}(\zeta_{\beta, j}, \zeta_{\beta, j+1}) \big)$,
as the  forcing notion
$${\bbP}^1_{1, \omega} / \dot{\bold{G}}_{\bbP^1_{1, k}(\beta)}  \ast \big(\prod_{j > k}\name{\Add}(\zeta_{\beta, j}, \zeta_{\beta, j+1})  \big) \ast \name{\bbP}^1_1(> \beta) \ast {\bbP}^1_{\omega} / {\dot G}_{\bbP^{1}_{1}} \ast \prod_{n<\omega} \name{\Add}(\kappa_n, \kappa_{n+1})$$
does not add new subsets to  $\bold V_{\zeta_{\beta, k}}$. Putting all things together,
we can conclude that
$\bbP^1 \ast \name{\bbP}^2 \ast \name{\bbR}$ can be written as
\[
\bbP^1 \ast \name{\bbP}^2 \ast \name{\bbR} \simeq \mathbb{L} \ast \name{\mathbb{K}} \ast \name{\mathbb{J}}
\]
where
\begin{itemize}
 \item[(6-1)] $\mathbb{L} \simeq \bbP^1_1(<\beta) \ast \bbP^1_{1, k}(\beta)$,

\item[(6-2)] $\mathbb{K}$ is forced by $1_{\mathbb{L}}$ to be of the form
   $$\name{\mathbb{K}} \simeq  \big(\prod_{j \leq k}\name{\Add}(\zeta_{\beta, j}, \zeta_{\beta, j+1}) \big) \ast \bigg(\name{\bbR}^{< \theta_{\delta+m}} \times \prod_{i \leq k} \name{\Add}(\zeta_{\beta, i}, \zeta_{\beta, i+1})\bigg),$$

 \item[(6-3)] $1_{\mathbb{L} \ast \name{\mathbb{K}}}$ forces $\name{\mathbb{J}} \simeq \name{\mathbb{J}}_1 \ast \name{\mathbb{J}}_2 \ast \name{\mathbb{J}}_3$,
 where
 \begin{enumerate}
\item[(a)] $\name{\mathbb{J}}_1  \simeq \ast\name{\bbP}^1_{1, \omega} / \dot{\bold{G}}_{\bbP^1_{1, k}(\beta)}  \ast \big(\prod_{j > k}\name{\Add}(\zeta_{\beta, j}, \zeta_{\beta, j+1})  \big) \ast \name{\bbP}^1_1(> \beta) \ast {\bbP}^1_{\omega} / {\dot G}_{\bbP^{1}_{1}}$,
 \item[(b)] $\name{\mathbb{J}}_2  \simeq \prod_{n<\omega} \name{\Add}(\kappa_n, \kappa_{n+1})$,

\item[(c)] $\name{\mathbb{J}}_3  \simeq \prod_{k<i<4} \name{\Add}(\zeta_{\beta, i}, \zeta_{\beta, i+1}) \times \name{\Add}(\zeta_{\beta, 4}, \theta_{\delta+m+1})   \times  \name{\bbR}^{> \theta_{\delta+m+1}}.$
 \end{enumerate}
  \end{itemize}
 By (4-1), the cardinal $\zeta_{\beta, k}$ remains measurable after forcing with $\mathbb{L}$, and  its measurability is indestructible under
 $\zeta_{\beta, k}$-closed forcing notions from $\bold V_\kappa$.
 It is also clear that $\mathbb{J}$ is forced to add no new subsets to $\zeta^+_{\beta, k}$. The lemma follows.
  \end{proof}

Let us first prove the desired statement in the extension by $\mathbb{L} \ast \name{\mathbb{K}}$.
We can write
\[
\mathbb{L} \ast \name{\mathbb{K}} \simeq \mathbb{L} \ast \name{\Add}(\zeta_{\beta, k}, \zeta_{\beta, k+1})
\ast \big(\prod_{j < k}\name{\Add}(\zeta_{\beta, j}, \zeta_{\beta, j+1})\big)  \ast \bbR^{< \theta_{\delta+m}} \footnote{Note that the Cohen forcing notions on both iterands of $\mathbb{K}$ are defined in the same universe, so the equivalence follows from the fact that
${\Add}(\zeta_{\beta, j}, \zeta_{\beta, j+1}) \times {\Add}(\zeta_{\beta, j}, \zeta_{\beta, j+1}) \simeq {\Add}(\zeta_{\beta, j}, \zeta_{\beta, j+1}).$ }
\]
Now the cardinal $\zeta_{\beta, k}$ remains measurable in the extension by $\mathbb{L} \ast \name{\Add}(\zeta_{\beta, k}, \zeta_{\beta, k+1})$,
hence by Lemma \ref{a4}(1), in the  model $\bold V[\bold{G}_{\mathbb{L} \ast \name{\Add}(\zeta_{\beta, k}, \zeta_{\beta, k+1})}]$,
\[
(*)_{\delta, n} \quad\quad\pcf(\beth_{n}(\theta_{\delta})^+ \cap \Reg)=\pcf(\zeta_{\beta, k}^+ \cap \Reg)=\zeta_{\beta, k+1}^+  \cap \Reg = \beth_{n+1}(\theta_{\delta})^+ \cap \Reg.
\]
As   the forcing notion $\prod_{j < k}\name{\Add}(\zeta_{\beta, j}, \zeta_{\beta, j+1})$ is $\zeta^+_{\beta, k-1}$-c.c.,\footnote{When $k=0$, the forcing becomes the trivial forcing, so the conclusion is still clear, and we can for example take $\zeta^+_{\beta, -1}=\aleph_1.$} by Lemma \ref{a4}(2), the above equalities continue to hold in
the extension $\bold V[\bold{G}_{\mathbb{L} \ast \big(\prod_{j \leq k}\name{\Add}(\zeta_{\beta, j}, \zeta_{\beta, j+1})\big)}]$.

Set $I_{\theta, n}=(\beta^+, \beth_{n}(\theta_{\delta})^+)$.  By (5-1), the forcing notion $\bbR^{< \theta_{\delta+m}}$ satisfies the $\beta^+=\zeta_{\beta, 0}^+$-c.c.,
hence by another application of
Lemma \ref{a4}(2),
$$\begin{array}{ll}
\pcf^{\bold V[\bold{G}_{\mathbb{L} \ast \name{\mathbb{K}}}]}(\beth_{n}(\theta_{\delta})^+ \cap \Reg) &\supseteq \pcf^{\bold V[\bold{G}_{\mathbb{L} \ast \name{\mathbb{K}}}]}(I_{\theta, n} \cap \Reg)\\
& \supseteq \pcf^{\bold V[\bold{G}_{\mathbb{L} \ast (\prod_{j \leq k}\name{\Add}(\zeta_{\beta, j}, \zeta_{\beta, j+1}))}]}(I_{\theta, n} \cap \Reg)\\ &\supseteq (\beta^+, \beth_{n+1}(\theta_{\delta})^+) \cap \Reg.
\end{array}$$
From this, we can conclude that the equality $(*)_{\delta, n}$ continues to hold in
in
the extension $\bold V[\bold{G}_{\mathbb{L} \ast \name{\mathbb{K}}}]$.
Finally as $\mathbb{J}$ has enough closure properties, the above equalities still continue to hold in
$\bold V^3.$
We are done.
\begin{remark}
The cardinal $\theta_0$ can be taken to be $(\theta_0+2)$-strong, so given any regular cardinal $\eta < \theta_0$,
if we force with Magidor forcing \cite{magidor} to change the cofinality of
$\theta_0$ into $\eta$,
then the above results
still hold
in the extension and in the resulting model,
$\theta_0$ is a singular cardinal of cofinality $\eta.$
\end{remark}


\begin{thebibliography}{99}
\bibitem{apter} Apter, Arthur W.; Hamkins, Joel David; \emph{Universal indestructibility}. Kobe J. Math. 16 (1999), no. 2, 119-130.

\bibitem{foreman} Foreman, Matthew; Woodin, W. Hugh; \emph{The generalized continuum hypothesis can fail everywhere}. Ann. of Math. (2) 133 (1991), no. 1, 1-35.

    \bibitem{jech}  Jech, Thomas; \emph{On the cofinality of countable products of cardinal numbers}. A tribute to Paul Erdos, 289-305, Cambridge Univ. Press, Cambridge, 1990.

\bibitem{laver}   Laver, Richard; \emph{Making the supercompactness of $\kappa$ indestructible under $\kappa$-directed closed forcing}. Israel J. Math. 29 (1978), no. 4, 385-388.

\bibitem{magidor} Magidor, Menachem;  \emph{Changing cofinality of cardinals}. Fund. Math. 99 (1978), no. 1, 61-71.

\bibitem{shelah} Shelah, Saharon; \emph{Cardinal arithmetic}. Oxford Logic Guides, 29. Oxford Science Publications. The Clarendon Press, Oxford University Press, New York, 1994. xxxii+481 pp. ISBN: 0-19-853785-9.

    \bibitem{kenta}  Tsukuura, Kenta; \emph{Prikry-type forcing and the set of possible cofinalities}. Tsukuba J. Math. 45 (2021), no. 2, 83-95.



\end{thebibliography}

\end{document}